\documentclass[12pt]{amsart}
\usepackage[headings]{fullpage}
\usepackage{amssymb,amsmath,amscd,bbm,tikz}
\usepackage{epic,eepic,epsfig}
\usepackage[all]{xy}
\usepackage{url}
\usepackage[bookmarks=true,%
    colorlinks=true,%
    linkcolor=blue,%
    citecolor=blue,%
    filecolor=blue,%
    menucolor=blue,%
    urlcolor=blue,%
    breaklinks=true]{hyperref}
\usepackage{slashed}    

\newtheorem{theorem}{Theorem}[section]
\theoremstyle{definition}

\newtheorem{lemma}[theorem]{Lemma}

\newtheorem{remark}[theorem]{Remark}
\newtheorem{corollary}[theorem]{Corollary}

\def\BN{\mathbbm N}
\def\BZ{\mathbbm Z}
\def\BQ{\mathbbm Q}
\def\BR{\mathbbm R}
\def\BC{\mathbbm C}

\def\calI{\mathcal I}

\def\calC{\mathcal C}

\def\tq{\tilde{q}}

\def\Li{\mathrm{Li}}
\def\PSL{\mathrm{PSL}}
\def\th{\theta}
\def\calS{\mathcal{S}}

\newcommand{\im}{\mathsf{i}}
\newcommand{\fad}[1]{\operatorname{\Phi}_{#1}}

\begin{document}
\title[Evaluation of state integrals at rational points]{Evaluation of state 
integrals at rational points}
\author{Stavros Garoufalidis}
\address{School of Mathematics \\
         Georgia Institute of Technology \\
         Atlanta, GA 30332-0160, USA \newline
         {\tt \url{http://www.math.gatech.edu/~stavros}}}
\email{stavros@math.gatech.edu}
\author{Rinat Kashaev}
\address{Section de Math\'ematiques, Universit\'e de Gen\`eve \\
2-4 rue du Li\`evre, Case Postale 64, 1211 Gen\`eve 4, Switzerland \newline
         {\tt \url{http://www.unige.ch/math/folks/kashaev}}}
\email{Rinat.Kashaev@unige.ch}
\thanks{
S.G. was supported in part by grant DMS-0805078 of the US National Science 
Foundation. R.K. was supported in part by the Swiss National Science Foundation.
}

\date{April 27, 2015}

\begin{abstract}
Multi-dimensional state-integrals of products of Faddeev's quantum 
dilogarithms arise frequently in Quantum Topology, quantum Teichm\"uller theory
and complex Chern--Simons theory. Using the quasi-periodicity property 
of the quantum dilogarithm, we evaluate 1-dimensional state-integrals at 
rational points and express the answer in terms of the Rogers dilogarithm,
the cyclic (quantum) dilogarithm and finite state-sums at roots of unity. 
We illustrate our results with the evaluation of the state-integrals 
of the $4_1$, $5_2$ and $(-2,3,7)$ pretzel knots at rational points.
\end{abstract}

\maketitle

\tableofcontents


\section{Introduction}
\label{sec.intro}

\subsection{State-integrals and their $q$-series}
\label{sub.intro}

State-integrals are multi-dimensional integrals of products of 
Faddeev's quantum dilogarithms. They appear in abundance in Quantum Topology,
quantum Teichm\"uller theory and in complex Chern--Simons theory. State
integrals were studied among others by 
Hikami~\cite{Hi}, Dimofte--Gukov--Lennels--Zagier~\cite{DGLZ}, 
Andersen--Kashaev~\cite{AK,AK:complex}, Kashaev--Luo--Vartanov \cite{KLV},
Dimofte~\cite{Dimofte:complex} and Dimofte--Garoufalidis~\cite{DG2}. 

In our previous paper~\cite{state-integrals}, we showed how to express 
1-dimensional state-integrals as a finite sum of products of $q$-series and 
$\tq$-series with integer coefficients, where the variables $q$ and $\tq$ are 
related by the modular transformation: $q=e^{2\pi\im\tau}$ and 
$\tq=e^{-2\pi\im/\tau}$.

In this paper we evaluate 1-dimensional state-integrals at rational points 
in terms of the Rogers dilogarithm, the cyclic (quantum) dilogarithm and 
truncated state-sums at roots of unity. Our formulas are syntactically 
similar with 
\begin{itemize}
\item[(a)]
the constant terms of the power series that appear in 
the Quantum Modularity Conjecture of Zagier~\cite{Za:QMF,GZ1},
\item[(b)]
the $1$-loop terms of the perturbation expansion of complex Chern--Simons
theory \cite{DG2},
\item[(c)]
the state-sums of quantum Teichm\"uller theory~\cite{Kashaev:6j,K95,K97} and 
also~\cite[Sec.6]{BB}.
\end{itemize} 
This is not a coincidence; it is one part of a story discussed
in detail in~\cite{GZ2}.

In order to keep our 
principle clear, we focus exclusively on 1-dimensional state-integrals, and we illustrate our results for the state-integrals of
$4_1$, $5_2$ and $(-2,3,7)$ pretzel knots. In a separate publication we will
discuss the evaluation of multi-dimensional state-integrals.

A 1-dimensional state-integral is an absolutely convergent integral of 
the form
\begin{equation}
\label{eq.IAB}
\calI_{A,B}(\mathsf{b})=
\int_{\BR + \im \epsilon} \fad{\mathsf{b}}(x)^B e^{-A \pi \im x^2} dx
\end{equation}
for a complex number $\mathsf{b}$ with $\mathsf{b}^2 \not\in \BR_{\leq 0}$. 
Here $A, B$ are natural
numbers satisfying $B > A > 0$ and $\fad{\mathsf{b}}(x)$ is Faddeev's quantum 
dilogarithm function~\cite{Faddeev,FK-QDL}. Few properties of this special
function are reviewed in Appendix \ref{sec.QDL}.

A numerical computation by the first author and Zagier~\cite{GZ2}
suggested the following formula for $\calI_{1,2}(1)$:
\begin{equation}
\label{eq.I211}
\calI_{1,2}(1)= \frac{e^{\pi\im/6}}{\sqrt{3}}
\left(e^{\frac{V}{2\pi}}-e^{-\frac{V}{2\pi}}\right)
\end{equation}
(and more generally for the Taylor coefficients of the analytic function 
$\calI_{1,2}(\mathsf{b})$ at $\mathsf{b}=1$), 
where $V=2\,\mathrm{Im}(\Li_2(e^{\pi \im/3}))=2.0298832\ldots$ is the 
volume of the $4_1$ knot. Understanding and proving the above identity led 
to the results of our paper. 

Our aim is to evaluate $\calI_{A,B}(\mathsf{b})$ when 
$\mathsf{b}^2=M/N$ for a pair of coprime natural numbers $M,N$. 
The content of our paper can be summarized in a diagram

$$
\begin{CD}
\{\text{state-integrals}\} @>>> \{\text{Nahm series}\} \\
@VVV                              @VVV   \\
\{\text{evaluations}\} @<<< \{\text{truncated Nahm series}\}
\end{CD}
$$
The top arrow was the content of our previous article~\cite{state-integrals}.
To recall the connection between state-integrals and $q$-series,
consider the integrand of the state-integral $\calI_{A,B}(\mathsf{b})$, 
shifted by $c_{\mathsf{b}}=\im (\mathsf{b}+\mathsf{b}^{-1})/2$:
\begin{equation}
\label{eq.fc}
\operatorname{f}(x-c_{\mathsf{b}})= \fad{\mathsf{b}}(x)^B e^{-A \pi \im x^2} \,.
\end{equation}
The quasi-periodicity of the quantum dilogarithm (see 
Equations~\eqref{eq.bshift}--\eqref{eq.tbshift}) implies that
$$
\operatorname{f}(x+ \im m \mathsf{b} + \im n \mathsf{b}^{-1})=
\operatorname{f}(x) \operatorname{g}^+_m(e^{2 \pi \mathsf{b} x},q_+) 
\operatorname{g}^-_n(e^{2 \pi \mathsf{b}^{-1} x},q_-)
$$
where $q_\pm=e^{2 \pi \im \mathsf{b}^{\pm 2}}$ and
\begin{equation}
\label{eq.gpm}
\operatorname{g}^\pm_k(x,q)= (-x)^{A k}\frac{q^{\frac{A}{2}k(k+1)}}{(qx;q)_k^B} \,.
\end{equation}
This gives rise to the series $G^\pm(x,q) \in \BZ[\![x,q]\!]$ 
defined by
\begin{align}
G^\pm(x,q)&=\sum_{k=0}^\infty \operatorname{g}^\pm_k(x,q) 
\end{align}
The $q$-series $G^{\pm}(1,q) \in \BZ[\![q]\!]$ are special $q$-hypergeometric
series of Nahm type and appear in the expression of the 
state-integral $\calI_{A,B}(b)$ as a sum of products of $q$-series and 
$\tq$-series, where $\tq=1/q_-$, see~\cite[Thm.1.1]{state-integrals}.

Throughout the paper, $(M,N)$ will denote an {\em admissible} pair, i.e.,
a pair of coprime positive integers. Consider the state-sum defined by
\begin{align}
G_{M,N}(x_+,x_-)&=\sum_{k=0}^{MN-1} \operatorname{g}^+_{kP}(x_+,\zeta_N^M)
\operatorname{g}^-_{kQ}(x_-,\zeta_M^N)
\end{align}
where $P,Q$ are integers that satisfy the equation $MP+NQ=1$
and $\zeta_N=e^{2 \pi i/N}$. When $x_+^N=x_-^M$, it follows from
Lemma~\ref{lem.gg} that $G_{M,N}(x_+,x_-)$ is independent of the choice of 
$P$ and $Q$. Observe that
$$
G_{1,N}(x_+,x_-)=G^+_N(x_+), \qquad G_{M,1}(x_+,x_-)=G^-_M(x_-) 
$$
where
\begin{align}
\label{eq.GN}
G^\pm_N(x)&=\sum_{k=0}^{N-1} \operatorname{g}^\pm_k(x,\zeta_N) \,.
\end{align}

\subsection{The Rogers and the cyclic dilogarithms}
\label{sub.dilog}

Recall the {\em Rogers dilogarithm}~\cite{Neumann:CCS,Goette}
\begin{equation}
\label{eq.rogers}
\operatorname{R}(z)=
\operatorname{Li_2}(z)+\frac12 \log(z)\log(1-z)-\frac{\pi^2}{6}
\end{equation}
and its extension as a multivalued function on the universal abelian cover
of $\BC\setminus\{0,1\}$.

The {\em cyclic (quantum) dilogarithm} $\operatorname{D}_N(x;q)$ is the 
$N$-th root of a polynomial in $x$ with constant term $1$ defined by
\begin{equation}
\label{eq.DN}
\operatorname{D}_N(x;q)=\prod_{k=1}^{N-1} (1-q^k x)^{k/N} \,.
\end{equation}
It appeared in~\cite[Eqn.C.3]{Kashaev:star} and 
\cite[Eqn.2.30]{Kashaev:quantum.hyper}, and its $N$-th power
is characterized among polynomials by the functional equation
$$
\frac{\operatorname{D}_N(\zeta_N x;\zeta_N)^N}{\operatorname{D}_N(x;\zeta_N)^N}
= \frac{(1-x)^N}{1-x^N}, \qquad \operatorname{D}_N(0)^N=1.
$$
It will be useful to introduce the following variant 
$\slashed{\operatorname{D}}_N$ defined by
\begin{equation}
\slashed{\operatorname{D}}_N(x;q)
=\prod_{k=1}^{N}(1-xq^k)^{k/N}=(1-xq^N)\operatorname{D}_N(x;q) \,.
\end{equation}

\subsection{Evaluation of state-integrals}
\label{sub.evaluation}

Our main theorem evaluates the state-integral at $\mathsf{b}^2=M/N$ 
in terms of the state-sums $G_{M,N}$, the Rogers dilogarithm and the
cyclic dilogarithm. 

Fix an admissible pair $(M,N)$, and define
\begin{equation}
\label{eq.bs}
\mathsf{b}= \sqrt{M/N}, \qquad \mathsf{s}= \sqrt{MN} \,.
\end{equation}
Let 
\begin{equation}
\label{eq.calS}
\calS =\{ w \,\, | \,\, 
\operatorname{g}(e^{2\pi\mathsf{s}w})=1, 
\,\, 0 < \mathsf{s}\, \mathrm{Im}(w)-\lambda < 1 \},
\end{equation}
where $\lambda$ is a generic real number such that 
\begin{equation}
\label{eq.lambda}
-(M+N)/2 < \lambda < 0
\end{equation}
and
\begin{equation}
\operatorname{g}(z)=
(-z)^{A} (1-z)^{-B} \in \BQ[z^{\pm 1}] \,.
\end{equation}
Note that if $w \in\calS$, then $e^{2\pi\mathsf{s}w}$ 
is an algebraic number with a fixed choice of $N$ and $M$-th roots. 

\begin{theorem}
\label{thm.1}
When $\mathsf{b}^2=M/N$ we have:
\begin{align}
\label{eq.Ib}
\calI_{A,B}(\mathsf{b}) &=
e^{\pi\im \frac{B+3A(M+N+1)^2-6 M N}{12 M N}} 
\mathsf{s}^{-1}
\sum_{w \in \calS}
\frac{e^{\frac{\im B}{2\pi \mathsf{s}^2}\operatorname{R}(z)} 
(1-z)^{\frac{(2N+1)(2M+1)}{4 M N}B}
}{z \operatorname{g}'(z) \slashed{\operatorname{D}}_N(\th_+,q_+)^{B} 
\slashed{\operatorname{D}}_M(\th_-,q_-)^{B}}  
G_{M,N}(\th_+,\th_-) \,,
\end{align}
where
\begin{equation}
\label{eq.conventions}
z=e^{2\pi\mathsf{s}w}, \quad
\th_+ =e^{2\pi\mathsf{b}w}=z^{1/N}, \quad \th_- =
e^{2\pi\mathsf{b}^{-1}w}=z^{1/M}, \quad 
q_+ = \zeta_N^M, \quad q_-=\zeta_M^N \,.
\end{equation}
\end{theorem}
Note that when $\operatorname{g}(z)=1$, we have
\begin{equation}
\label{eq.gp}
\operatorname{g}'(z)=A z^{-1} + B (1-z)^{-1} \,.
\end{equation}

\begin{corollary}
\label{cor.1}
For $M=1$ we obtain that
\begin{equation}
\label{eq.IbN}
\calI_{A,B}(\mathsf{b})=
e^{\pi\im \frac{B+3 A (N+2)^2-6N}{12 N}} 
\frac{1}{\sqrt{N}}
\sum_{w \in \calS}
e^{\frac{\im B}{2\pi N}\operatorname{R}(z)}
\frac{(1-z)^{\frac{2N+3}{4N}B}}{
(A+B z/(1-z)) \slashed{\operatorname{D}}_N(\th_+,q_+)^{B}}  
G^+_N(\th_+) \,.
\end{equation}
When $M=N=1$ we obtain that
\begin{equation}
\label{eq.Ib1}
\calI_{A,B}(1)=
e^{\pi\im \frac{B+3 A -6}{12}} 
\sum_{w \in \calS}
e^{\frac{\im B}{2\pi}\operatorname{R}(z)}
\frac{(1-z)^{\frac{B}{4}}}{(A+B z/(1-z))} \,.
\end{equation}
\end{corollary}

Let us denote
\begin{equation}
\label{eq.ex}
\mathsf{e}(x)=e^{2 \pi \im x} \,.
\end{equation}

\begin{corollary}
\label{cor.41.b=1}
When $M=N=1$ and $(A,B)=(1,2)$, we choose $\lambda$ to be a negative real
number near zero,
\begin{align*}
\operatorname{g}(z) &=-z(1-z)^{-2} \\
\calS &= \{\im /6, 5 \im /6 \} \\
z_\pm &= \mathsf{e}(\pm1/6)\\
e^{\pi\im \frac{B+3 A -6}{12}} &= \mathsf{e}\left(-\frac{1}{24}\right) \\
(e^{\frac{\im B}{2\pi}\operatorname{R} (z_+)} ,
e^{\frac{\im B}{2\pi}\operatorname{R} (z_-)} )
&= \left(e^{-C} \mathsf{e}\left(-\frac{1}{24}\right),
- e^{C} \mathsf{e}\left(-\frac{1}{24}\right) 
\mathsf{e}\left(\frac{1}{3}\right) \right)
\\
\frac{(1-z_\pm)^{\frac{B}{4}}}{(A+B z_\pm/(1-z_\pm))}
&=\frac{1}{\sqrt{3}} \mathsf{e}\left(\mp \frac{1}{3}\right)
\end{align*}
where $C=V/(2\pi)$ and $V$ is the volume of the $4_1$ knot. 
When computing the Rogers dilogarithm of $z_\pm$, keep in mind that we
use the branches of the logarithm
$\log z_+= 2 \pi \im/6$ and $\log z_-= 10 \pi \im/6$ dictated by
Equations~\eqref{eq.calS} and~\eqref{eq.lambda}.

The above computation, combined with Equation~\eqref{eq.Ib1} implies 
Equation~\eqref{eq.I211}. As was already mentioned, the proof
of this equation was a main motivation for the results of our paper.
\end{corollary}

We now make few remarks about the number-theoretic, analytic and geometric properties of 
Equation~\eqref{eq.Ib}. 

\begin{remark}
\label{rem.algebraic}
It~\cite{GZ1} (see also~\cite{DG2}) it was observed that although 
$(G^+(\th_+))^N$ and $(\slashed{\operatorname{D}}_N(\th_+,\zeta_N))^N$ 
lie in the field $F_{G,N}=\BQ(\th_+,\zeta_N)$, their ratio lies in the 
smaller field $\BQ(z,\zeta_N)$ (where $z$
satisfies $\operatorname{g}(z)=1$) which is an extension of $\BQ(z)$  
by $\zeta_N$. In particular, the above
mentioned ratio is independent of the choice of the $N$-th root of $z$. 
\end{remark}

\begin{remark}
\label{rem.multivalued}
Although $\slashed{\operatorname{D}}_N$ is a multivalued function, the
sum in Equation~\eqref{eq.Ib} is well-defined. This is a consequence
of Theorem~\ref{thm.2} below and the fact that the quantum dilogarithm is 
a meromorphic function.
\end{remark}

\begin{remark}
\label{rem.parabolic}
When the state-integral is associated with a cusped hyperbolic manifold $M$, 
the set $\calS$ is often in bijection with the set of nonabelian parabolic 
$\PSL(2,\BC)$ representations of $M$. Under such a bijection, the Rogers
dilogarithm matches with the complex volume, and the value of $g'(z)$ matches
with the value of the 1-loop invariant of~\cite{DG}, suitably normalized.
For an illustration, see Section~\ref{sub.pretzel}.
\end{remark}

\begin{remark}
\label{rem.QHG}
When the state-integral is associated with a cusped hyperbolic manifold $M$
and the identification of Remark~\ref{rem.parabolic} is available,
one can identify Equation~\eqref{eq.Ib} with a sum of invariants of $M$
parametrized by nonabelian parabolic $\PSL(2,\BC)$ representations of $M$.
Such invariants appear in Quantum Hyperbolic Geometry--see~\cite{K97} and
also~\cite{BB}. The invariants of Quantum Hyperbolic Geometry are defined 
up to multiplication by an $N$-root of unity. However, Equation~\eqref{eq.Ib}
gives a well-defined {\em relative} choice of the $N$-th roots of unity.
This is a consequence of the meromorphicity of the quantum dilogarithm. 
\end{remark}

\begin{remark}
\label{rem.taylor}
As we already mentioned above, a numerical computation by the first
author and Zagier suggests an explicit formula for the Taylor series
of $\calI_{1,2}(\mathsf{b})$ at $\mathsf{b}=1$ in terms of the asymptotics
of the Kashaev invariant at $q=1$. We expect that the Taylor series
of state-integrals at $\mathsf{b}=\sqrt{M/N}$ can be expressed in terms of
the loop invariants of Garoufalidis--Dimofte~\cite{DG2}. We plan to study
this in a later publication.
\end{remark}

Theorem \ref{thm.1} follows from a lemma from complex analysis
regarding integrals of quasi-periodic functions~\ref{lem.main}. This 
lemma is used twice, once to evaluate the quantum dilogarithm in terms of the 
cyclic dilogarithm, and another time to evaluate the state-integral 
$\calI_{A,B}(\mathsf{b})$.

\subsection{The quantum dilogarithm at roots of unity}
\label{sub.evalphi}

We fix an admissible pair $(M,N)$. Recall 
$\mathsf{b}$ and $\mathsf{s}$ from Equation~\eqref{eq.bs}.
Let $\Li_2(z)=\sum_{n=1}^\infty \frac{z^n}{n^2}$ denote the Euler dilogarithm,
defined for $|z|<1$ and analytically continued as a multivalued
function on $\BC\setminus\{0,1\}$.

\begin{theorem}
\label{thm.2}
We have:
\begin{equation}
\label{eq.Phieval}
\fad{\mathsf{b}}\!\left(\frac{z}{2\pi\mathsf{s}}-c_{\mathsf{b}}\right)=
\frac{e^{\frac{\im}{2\pi\mathsf{s}^2}\operatorname{Li_2}(e^{z})}
\left(1-e^z\right)^{1+\frac{\im z}{2\pi\mathsf{s}^2}}}{
\operatorname{D}_N(e^{z/N};q_+)\operatorname{D}_M(e^{z/M};q_-)} \,.
\end{equation}
\end{theorem}

It is remarkable that the left-hand side is a meromorphic function of $z$
whereas the right hand side is assembled out of multivalued functions
of $z$.

In particular when $M=1$, we obtain that
\begin{equation}
\label{eq.evalphiN}
\fad{\mathsf{b}}(x-c_{\mathsf{b}})= \frac{e^{-\frac{1}{2 \pi \im N}\Li_2(z^N)}}{
\operatorname{D}_N(z)}
(1-z^N)^{1+ \frac{\im x}{\sqrt{N}}}, \qquad z=e^{2 \pi \mathsf{b} x} \,,
\end{equation}
and when $M=N=1$, we obtain that
\begin{equation}
\label{eq.phi1}
\fad{1}(x)=\exp\left( \frac{\im}{2 \pi} \left( \Li_2(e^{2 \pi x})+2 \pi x
\log(1-e^{2 \pi x})\right)\right) \,.
\end{equation}

By using the equality
\begin{equation}
\frac{\fad{\mathsf{b}}\!\left(\frac{z}{2\pi\mathsf{s}}-c_{\mathsf{b}}\right)}
{\fad{\mathsf{b}}\!\left(\frac{z}{2\pi\mathsf{s}}+c_{\mathsf{b}}\right)}=
\left(1-e^{z/N}\right)\left(1-e^{z/M}\right)
\end{equation}
we also have
\begin{equation}\label{eq.Phieval+}
\fad{\mathsf{b}}\!\left(\frac{z}{2\pi\mathsf{s}}+c_{\mathsf{b}}\right)=
\frac{e^{\frac{\im}{2\pi\mathsf{s}^2}\operatorname{Li_2}(e^{z})}
\left(1-e^z\right)^{1+\frac{\im z}{2\pi\mathsf{s}^2}}}{
\slashed{\operatorname{D}}_N(e^{z/N};q_+)
\slashed{\operatorname{D}}_M(e^{z/M};q_-)},
\end{equation}

\begin{remark}
\label{rem.ORV}
The cyclic dilogarithm is in a sense a radial limit of the generating series
$$
M(x,q)=\prod_{n=1}^\infty \frac{1}{(1-x q^n)^n}
$$
where $M(x,q)$ is the McMahon generating series of 3-dimensional
{\em plane partitions}; see \cite{top-vertex} and \cite[Sec.2.1]{ORV}. 
The latter appear in M-theory and mirror symmetry. It would be
interesting and useful to understand a precise relation between 
the function $\operatorname{D}_N$ and plane partitions. 
\end{remark}

\subsection*{Acknowledgments}
The paper originated as an attempt to prove an identity conjectured by
joint work of Zagier and the first author. We wish to thank Tudor Dimofte
and especially Don Zagier for enlightening conversations and for a 
generous sharing of their ideas. The paper was conceived during a conference 
in Vietnam in 2013, and largely completed in Geneva in 2014 at the Confucius 
Institute of the University of Geneva, and during a conference in Quantum 
Topology in Magnitogorsk, Russia in 2014 and during an Oberwolfach workshop 
in August 2014. The authors wish to thank the organizers of the conferences 
for their hospitality.

S.G. was supported in part by grant DMS-0805078 of the US National Science 
Foundation. R.K. was supported in part by the Swiss National Science Foundation.


\section{Integrals of quasi-periodic functions}
\label{sec.int.quasi}

\subsection{Some lemmas from complex analysis}
\label{sub.two.lemmas}

\begin{lemma}
\label{lem1}
Let $\operatorname{f}\colon \mathcal{U}\to\mathbb{C}$ be an analytic 
function satisfying the functional equation
\begin{equation}
\label{per}
\operatorname{f}(z-a)\operatorname{f}(z+a)=\operatorname{f}(z)^2,
\end{equation}
with some fixed $a\in\mathbb{C}\setminus\{0\}$, the domain 
$\mathcal{U}\subset\mathbb{C}$ being a translationally invariant open set, 
$\mathcal{U}=a+\mathcal{U}$,
and $\mathcal{C}\subset \mathcal{U}$ an oriented path such that 
$\operatorname{f}(z)(\operatorname{f}(z)-\operatorname{f}(z+a))\ne0$ for 
all $z\in \mathcal{C}$. Then
\begin{equation}
\label{residue1}
\int_{\mathcal{C}}\operatorname{f}(z)\operatorname{d}\!z = 
\left(\int_{\mathcal{C}}- \int_{a+\mathcal{C}}\right)
\frac{\operatorname{f}(z)}{1-\operatorname{f}(z+a)/\operatorname{f}(z)}
\operatorname{d}\!z.
\end{equation}
\end{lemma}

\begin{proof}
We have
\begin{multline}
\left(\int_{\mathcal{C}}- \int_{a+\mathcal{C}}\right)
\frac{\operatorname{f}(z)}{1-\operatorname{f}(z+a)/
\operatorname{f}(z)}\operatorname{d}\!z 
=\int_{\mathcal{C}}\frac{\operatorname{f}(z)}{
1-\operatorname{f}(z+a)/\operatorname{f}(z)}\operatorname{d}\!z\\ 
-\int_{a+\mathcal{C}}\frac{\operatorname{f}(z)}{
1-\operatorname{f}(z)/\operatorname{f}(z-a)}\operatorname{d}\!z
=\int_{\mathcal{C}}\frac{\operatorname{f}(z)-\operatorname{f}(z+a)}{
1-\operatorname{f}(z+a)/\operatorname{f}(z)}\operatorname{d}\!z= 
\int_{\mathcal{C}}\operatorname{f}(z)\operatorname{d}\!z .
\end{multline}
\end{proof}

Given a rational function $\operatorname{r}(z)\in\mathbb{C}(z)$, 
a nonzero complex number $q\in\mathbb{C}\setminus\{0\}$ and an integer
$k$ we define $\operatorname{r}_k(z;q)\in\mathbb{C}(z)$, $k\in\mathbb{Z}$,  
by:
\begin{equation}
\label{rk}
\frac{\operatorname{r}_{k+1}(z;q)}{\operatorname{r}_k(z;q)}
=\operatorname{r}(zq^k),\quad \operatorname{r}_0(z;q)=1.
\end{equation}
Note that
\begin{equation}
\operatorname{r}_{1}(z;q)=\operatorname{r}(z)
\end{equation}
and 
\begin{equation}
\operatorname{r}_{k+l}(z;q)
=\operatorname{r}_k(zq^l;q)\operatorname{r}_l(z;q)
=\operatorname{r}_l(zq^k;q)\operatorname{r}_k(z;q)\,,
\end{equation}
for all integers $k,l$. In particular, 
in the case of roots of unity this implies a quasi-periodicity property:
\begin{equation}
\operatorname{r}_{k+N}(z;q)=\operatorname{r}_k(z;q)\operatorname{r}_N(z;q),
\quad \text{if}\quad q^N=1,
\end{equation}
and the invariance property
\begin{equation}
\operatorname{r}_{N}(zq;q)=\operatorname{r}_N(z;q),\quad \text{if}\quad q^N=1.
\end{equation}

Let $(M,N)$ be an admissible pair, and recall $\mathsf{b}$ and 
$\mathsf{s}$ from Equation~\eqref{eq.bs}.
Choose two integers $P$ and $Q$ which satisfy the equation $MP+NQ=1$. 
Let $\operatorname{f}(z)$ be a meromorphic function and 
$\operatorname{g^\pm}(z)\in\mathbb{C}(z)$ two rational functions such that
\begin{equation}
\frac{\operatorname{f}(z+\mathsf{b}^{\pm1}\im)}{\operatorname{f}(z)}=
\operatorname{g^\pm}\!\left(e^{2\pi\mathsf{b}^{\pm1}z}\right).
\end{equation}
Then, for $k \in \BZ$ we have
\begin{equation}
\operatorname{g^\pm}_k\!\left(e^{2\pi\mathsf{b}^{\pm1}z};q_{\pm}\right)
=\frac{\operatorname{f}(z+\mathsf{b}^{\pm1}\im k)}{\operatorname{f}(z)}\,,
\end{equation}
and, in particular,
\begin{equation}
\operatorname{g^+}_{N}\!\left(e^{2\pi\mathsf{b}z};q_+\right)
= \frac{\operatorname{f}(z+\mathsf{b}\im N)}{\operatorname{f}(z)}
=\frac{\operatorname{f}(z+\mathsf{s}\im )}{\operatorname{f}(z)}
=\frac{\operatorname{f}(z+\mathsf{b}^{-1}\im M)}{\operatorname{f}(z)}=
\operatorname{g^{--}}_{M}\!\left(e^{2\pi\mathsf{b}^{-1}z};q_-\right).
\end{equation}
Define
\begin{align}
\label{gx}\operatorname{g}\!\left(x\right) &= 
\operatorname{g^+}_{N}\!\left(x^{\frac1N};q_+\right)= 
\operatorname{g^{--}}_{M}\!\left(x^{\frac1M};q_-\right)
\\
\label{sxy}\operatorname{S}(x,y)&= 
\sum_{k=0}^{\mathsf{s}^2-1}\operatorname{g^+}_{kP}\!
\left(x;q_+\right)\operatorname{g^{--}}_{kQ}\!
\left(y;q_-\right)
\end{align}

\begin{lemma}
\label{lem.gg}
\rm{(a)}
We have $\operatorname{g}\!\left(x\right) \in\mathbb{C}(x)$ and
$\operatorname{S}(x,y) \in\mathbb{C}(x,y)$. 
\newline
\rm{(b)}
The function $\operatorname{S}(x,y) $ is independent of $P$ and $Q$ 
provided $x^N=y^M$.
\end{lemma}

\begin{proof}
Since
\begin{equation}
\operatorname{g^+}_{N}\!\left(xq_+;q_+\right)
=\operatorname{g^+}_{N}\!\left(x;q_+\right),\qquad 
\operatorname{g^{--}}_{M}\!\left(xq_-;q_-\right)
=\operatorname{g^{--}}_{M}\!\left(x;q_-\right),
\end{equation}
it follows that $\operatorname{g}\!\left(x\right) \in\mathbb{C}(x)$, and
consequently, $\operatorname{S}(x,y) \in\mathbb{C}(x,y)$.
 
For (b), let $P', Q'$ be another pair satisfying the equation $MP'+NQ'=1$. 
Then there exists an integer $R$ such that $P'=P+RN$ and $Q'=Q-RM$. Denoting 
the function \eqref{sxy} with $P$ and $Q$ replaced by $P'$ and $Q'$ as 
$\operatorname{S}'(x,y)$, we have
\begin{align*}
 \operatorname{S}'(x,y) 
&=\sum_{k=0}^{\mathsf{s}^2-1}\operatorname{g^+}_{kP'}\!
\left(x;q_+\right)\operatorname{g^{--}}_{kQ'}\!
\left(y;q_-\right)
=\sum_{k=0}^{\mathsf{s}^2-1}\operatorname{g^+}_{kP+kRN}\!
\left(x;q_+\right)\operatorname{g^{--}}_{kQ-kRM}\!
\left(y;q_-\right)\\
&=\sum_{k=0}^{\mathsf{s}^2-1}\operatorname{g^+}_{kP}\!
\left(x;q_+\right) \operatorname{g^+}_{N}\!
\left(x;q_+\right)^{kR}\operatorname{g^{--}}_{kQ}\!
\left(y;q_-\right)\operatorname{g^{--}}_{M}\!
\left(y;q_-\right)^{-kR}\\
&=\sum_{k=0}^{\mathsf{s}^2-1}\operatorname{g^+}_{kP}\!
\left(x;q_+\right) \left(\frac{\operatorname{g^+}_{N}\!
\left(x;q_+\right)}{\operatorname{g^{--}}_{M}\!
\left(y;q_-\right)}\right)^{kR}\operatorname{g^{--}}_{kQ}\!
\left(y;q_-\right)
=\sum_{k=0}^{\mathsf{s}^2-1}\operatorname{g^+}_{kP}\!
\left(x;q_+\right)\operatorname{g^{--}}_{kQ}\!
\left(y;q_-\right) \\
&=\operatorname{S}(x,y)
\end{align*}
where we used the second equality in \eqref{gx}.
\end{proof}
For a complex number $x$, we denote 
$\mathcal{C}_x= x\im/\mathsf{s}+\mathbb{R}\subset \BC$.
\begin{lemma}
\label{lem.main} Let $\operatorname{f}(z)$, 
$\operatorname{g}\!\left(z\right)$, $\operatorname{S}\!\left(x,y\right)$ 
be as above and $\lambda$ a real number in general position such that
the form $\operatorname{f}(z)\operatorname{d}\!z$ is absolutely integrable 
along 
$\mathcal{C}_{\lambda}$,
\begin{equation}
\label{cii}
\lim_{x\to\pm\infty}
\sup_{y\in \left[{\lambda\over \mathsf{s}},{\lambda\over \mathsf{s}}+\mathsf{s}\right]}|
\operatorname{f}(x+\im y)|=0,
\end{equation}
and
\begin{equation}
\label{ciii}
\operatorname{g}\!\left(0\right)\ne 1\ne \operatorname{g}\!\left(\infty\right).
\end{equation}
Then  the following equalities hold
\begin{align}
\label{eq.main}
\int_{\mathcal{C}_{\lambda}}\operatorname{f}(z)\operatorname{d}\!z &=
\left(\int_{\mathcal{C}_{\lambda}}-\int_{\mathcal{C}_{\lambda+1}}\right)
\frac{\operatorname{f}(z)\operatorname{S}\!
\left(e^{2\pi\mathsf{b}z},e^{2\pi\mathsf{b}^{-1}z}\right)}{1-\operatorname{g}\!
\left(e^{2\pi\mathsf{s}z}\right)}\operatorname{d}\!z \\ \notag
& =2\pi\im\sum_{0<\mathsf{s}\mathrm{Im}\alpha-\lambda<1}\operatorname{Res}_{z=\alpha}
\frac{\operatorname{f}(z)\operatorname{S}\!
\left(e^{2\pi\mathsf{b}z},e^{2\pi\mathsf{b}^{-1}z}\right)}{1-\operatorname{g}\!
\left(e^{2\pi\mathsf{s}z}\right)} \,.
\end{align}
\end{lemma}

\begin{proof}
Let us derive the first equality. We denote
\begin{equation}
q_\pm= e^{2\pi\im\mathsf{b}^{\pm2}},
\end{equation}
so that we have
\begin{equation}
q_+=e^{2\pi\im\frac MN},\quad q_-=e^{2\pi\im\frac NM}.
\end{equation}
We also have
\begin{equation}
\operatorname{f}(z+k\mathsf{b}^{\pm1}\im)=\operatorname{f}(z)
\operatorname{g^\pm}_{k}\!\left(e^{2\pi\mathsf{b}^{\pm1}z};q_\pm\right), 
\quad \forall k\in\mathbb{Z}.
\end{equation}
In particular,
\begin{equation}
\operatorname{f}(z+\mathsf{s}\im)=\operatorname{f}(z+N\mathsf{b}\im)=
\operatorname{f}(z+M\mathsf{b}^{-1}\im)=
\operatorname{f}(z)\operatorname{g}\!\left(e^{2\pi\mathsf{s}z}\right).
\end{equation}
The function
\begin{equation}\label{hf}
\operatorname{h}(z)= \operatorname{f}\!
\left(\frac z{2\pi\mathsf{s}}\right)
\end{equation}
has the properties
\begin{equation}\label{g+g-}
\frac{\operatorname{h}(z+k2\pi \im)}{\operatorname{h}(z)}=
\operatorname{g^+}_{Pk}\!\left(e^{z/N};q_+\right)
\operatorname{g^{--}}_{Qk}\!\left(e^{z/M};q_-\right),\quad \forall k\in\mathbb{Z}.
\end{equation}
From equation~\eqref{g+g-}, it follows that
\begin{equation}\label{Sh}
\sum_{k=0}^{\mathsf{s}^2-1}
\frac{\operatorname{h}(z+k2\pi\im)}{\operatorname{h}(z)}=
\operatorname{S}\!\left(e^{z/N},e^{z/M}\right).
\end{equation}
As $\lambda$ is generic, the contour $\mathcal{C}_\lambda$ satisfies the 
conditions of Lemma~\ref{lem1} with $a=\mathsf{s}\im$.  Thus, we write
\begin{multline}\label{pro1}
\int_{\mathcal{C}_\lambda}\operatorname{f}(z)\operatorname{d}\!z=
\!\left(\int_{\mathcal{C}_\lambda}-\int_{\mathsf{s}\im+\mathcal{C}_\lambda}\right)
\frac{\operatorname{f}(z)}{1-\operatorname{g}\!\left(e^{2\pi\mathsf{s}z}\right)}
\operatorname{d}\!z \\
=\left(\int_{\mathcal{C}_{\lambda 2\pi\mathsf{s}}}
-\int_{\mathcal{C}_{(\lambda+\mathsf{s}^2) 2\pi\mathsf{s}}}\right)
\frac{\operatorname{h}(z)}{1-\operatorname{g}\!\left(e^{z}\right)}
\frac{\operatorname{d}\!z}{2\pi\mathsf{s}} \\
=\sum_{k=0}^{\mathsf{s}^2-1}\left(\int_{\mathcal{C}_{(\lambda +k)2\pi\mathsf{s}}}
-\int_{\mathcal{C}_{(\lambda+k+1) 2\pi\mathsf{s}}}\right)
\frac{\operatorname{h}(z)}{1-\operatorname{g}\!\left(e^{z}\right)}
\frac{\operatorname{d}\!z}{2\pi\mathsf{s}} \\
=\sum_{k=0}^{\mathsf{s}^2-1}\left(\int_{\mathcal{C}_{\lambda 2\pi\mathsf{s}}}
-\int_{\mathcal{C}_{(\lambda+1) 2\pi\mathsf{s}}}\right)
\frac{\operatorname{h}(z+k2\pi\im)}{1-\operatorname{g}\!
\left(e^{z}\right)}\frac{\operatorname{d}\!z}{2\pi\mathsf{s}} \\
=\left(\int_{\mathcal{C}_{\lambda 2\pi\mathsf{s}}}-
\int_{\mathcal{C}_{(\lambda+1) 2\pi\mathsf{s}}}\right)
\frac{\operatorname{h}(z)\operatorname{S}\!
\left(e^{z/N},e^{z/M}\right)}{1-\operatorname{g}\!\left(e^{z}\right)}
\frac{\operatorname{d}\!z}{2\pi\mathsf{s}} \\
=\left(\int_{\mathcal{C}_{\lambda}}-\int_{\mathcal{C}_{\lambda+1}}\right)
\frac{\operatorname{f}(z)\operatorname{S}\!
\left(e^{2\pi\mathsf{b}z},e^{2\pi\mathsf{b}^{-1}z}\right)}{1-\operatorname{g}\!
\left(e^{2\pi\mathsf{s}z}\right)}\operatorname{d}\!z.
\end{multline}

The second equality in Equation~\eqref{eq.main} follows from the residue 
theorem. To justify its application note that $\mathcal{C}_{\lambda}-
\mathcal{C}_{\lambda+1}$ union two vertical segments is the boundary of a disk.
Conditions~\eqref{cii}, 
\eqref{ciii} imply that for sufficiently large  negative or positive $x$ 
the integrand in last part of \eqref{pro1} is regular on the vertical segment 
$\calC'= x+\im[\lambda,\lambda+1]/\mathsf{s}$ and the following estimate holds 
\begin{align*}
\left|\int_{\calC'}{\operatorname{f}(z)\operatorname{S}\!
\left(e^{2\pi\mathsf{b}z},e^{2\pi\mathsf{b}^{-1}z}\right)\over 1-\operatorname{g}\!
\left(e^{2\pi\mathsf{s}z}\right)}\operatorname{d}\!z\right|
 & \le
\sum_{k=0}^{\mathsf{s}^2-1}\int_{\lambda}^{\lambda+1}
\left |{\operatorname{f}\left(x+\im {y+k\over\mathsf{s}}\right)\over1-\operatorname{g}\!
\left(e^{2\pi(\mathsf{s}x+\im y)}\right)} \right |{\operatorname{d}\!y\over \mathsf{s}}\\
 & \le
{1\over\mathsf{s}}\sum_{k=0}^{\mathsf{s}^2-1}
\sup_{y\in [\lambda,\lambda+1]}\left |{\operatorname{f}\left(x+\im {y+k\over\mathsf{s}}\right)\over1-\operatorname{g}\!
\left(e^{2\pi(\mathsf{s}x+\im y)}\right)} \right |\\
& \le
{1\over\mathsf{s}}\sup_{y\in [\lambda,\lambda+1]}{1\over\left|1-\operatorname{g}\!
\left(e^{2\pi(\mathsf{s}x+\im y)}\right) \right |} 
\sum_{k=0}^{\mathsf{s}^2-1}
\sup_{t\in [\lambda,\lambda+1]}\left |\operatorname{f}\left(x+\im {t+k\over\mathsf{s}}\right)\right| \\
& \le\mathsf{s}{\sup_{z\in x+\im\left[{\lambda\over\mathsf{s}},{\lambda\over\mathsf{s}}+\mathsf{s}\right]}\left|\operatorname{f}(z)\right|\over\inf_{|w|=e^{2\pi\mathsf{s}x}}\left|1-\operatorname{g}\!
\left(w\right) \right |}
\xrightarrow{|x|\to\infty} 0
\end{align*}
\end{proof}

\subsection{Applications to 1-dimensional state-integrals: 
proof of Theorem~\ref{thm.1}}
\label{sub.application}

In this section we prove Theorem~\ref{thm.1}. 
Fix integers $A$ and $B$ with $B>A>0$. The values of particular 
interest are $(A,B)=(1,2)$ and $(1,3)$ which correspond to the state-integrals
of the knots $4_1$ and $5_2$ respectively. For the
$4_1$ knot, see ~\cite[Eqn.38]{AK} and \cite[Eqn.47]{KLV} and 
\cite{state-integrals}. For the $5_2$ knot, see ~\cite[Eqn.53]{KLV}.
For the remaining values of $(A,B)$, although the 1-dimensional state integral
makes sense, and it can be analyzed using our methods, there is no 
corresponding knot that we know of.

If
\begin{equation}
\operatorname{f}(z-c_{\mathsf{b}})=\fad{\mathsf{b}}(z)^Be^{-A\pi\im z^2},
\quad c_{\mathsf{b}}=(\mathsf{b}+\mathsf{b}^{-1})\im/2,
\end{equation}
then
\begin{subequations}
\begin{align}
\label{gxAB}
\operatorname{g}\!\left(x\right) &= 
\left(-x \right)^A\! \left(1-x\right)^{-B}
\\
\operatorname{g^\pm}(x)&=\operatorname{g}\!\left(q_\pm x\right)
\\
\operatorname{g^\pm}_n(x;q_\pm)&=\left(-x \right)^{An}\!
q_\pm^{\frac{A}{2}n(n+1)}\!\left(q_\pm x;q_\pm\right)_{n}^{-B},\qquad 
\forall n\in\mathbb{Z}.
\end{align}
\end{subequations}
Observe that $\operatorname{f}(z)$ is non-vanishing and absolutely 
integrable along the line $\mathcal{C}_{\lambda}$ if
\begin{equation}
-(M+N)/2<\lambda<0,
\end{equation}
and $\operatorname{f}(z+\mathsf{s}\im)\ne \operatorname{f}(z)$ if 
$\lambda$ is in general position. 
By using~\eqref{eq.Phieval+}, we obtain that
\begin{align}
\label{fzAB}
\operatorname{f}(z)&=\fad{\mathsf{b}}(z+c_{\mathsf{b}})^B
e^{-A\pi\im (z+c_{\mathsf{b}})^2} \\ \notag 
&
=\frac{e^{\frac{\im B}{2\pi\mathsf{s}^2}\operatorname{Li_2}(e^{2\pi\mathsf{s}z})}
\left(1-e^{2\pi\mathsf{s}z}\right)^{B+B\im z\mathsf{s}^{-1}}}{
\slashed{\operatorname{D}}_N(e^{2\pi\mathsf{b}z};q_+)^B
\slashed{\operatorname{D}}_M(e^{2\pi\mathsf{b}^{-1}z};q_-)^B}
e^{-A\pi\im (z+c_{\mathsf{b}})^2}.
\end{align}
By using the identity
\begin{align*}
e^{-A\pi\im (z+c_{\mathsf{b}})^2} &=
e^{-A\pi\im \left(c_{\mathsf{b}}+\frac{\im}{2\mathsf{s}}\right)^2}
e^{A(2\pi\mathsf{s}z-\pi\im) (1-4\mathsf{s}\im c_{\mathsf{b}} -2\mathsf{s}\im z)\mathsf{s}^{-2}/4}
\\
&=e^{A\pi\im \frac{(M+N+1)^2}{4 M N}}
e^{A(2\pi\mathsf{s}z-\pi\im) (1-4\mathsf{s}\im c_{\mathsf{b}} -2\mathsf{s}\im z)\mathsf{s}^{-2}/4}
\end{align*}
and \eqref{gxAB}, we can rewrite \eqref{fzAB} in the form
\begin{align}
\label{eq.fzval}
\operatorname{f}(z)&=
\frac{e^{\frac{\im B}{2\pi \mathsf{s}^2}\operatorname{R}(e^{2\pi\mathsf{s}z})
(1-e^{2\pi\mathsf{s}z})^{\frac{(2N+1)(2M+1)}{4 M N}B}
}}{
\slashed{\operatorname{D}}_N(e^{2\pi\mathsf{b}z};q_+)^B
\slashed{\operatorname{D}}_M(e^{2\pi\mathsf{b}^{-1}z};q_-)^B} 
\,\, e^{\pi\im \frac{B+3A(M+N+1)^2}{12 M N}} 
\operatorname{g}\!\left(e^{2\pi\mathsf{s}z}
\right)^{\frac{1-4\mathsf{s}\im c_{\mathsf{b}} -2\mathsf{s}\im z}{4\mathsf{s}^{2}}},
\end{align}
where $R(x)$ is the Rogers dilogarithm~\eqref{eq.rogers}.

It is easy to see that the only singularities in Equation~\eqref{eq.main} 
are simple poles that come from solutions to the equation 
$1-\operatorname{g}\! \left(e^{2\pi\mathsf{s}z}\right)=1$.
Moreover, if $z=\alpha$ is a solution with 
$0<\mathsf{s}\, \mathrm{Im}\alpha-\lambda<1$,
then
$$
2\pi\im \operatorname{Res}_{z=\alpha}
\frac{\operatorname{f}(z)\operatorname{S}\!
\left(e^{2\pi\mathsf{b}z},e^{2\pi\mathsf{b}^{-1}z}\right)}{1-\operatorname{g}\!
\left(e^{2\pi\mathsf{s}z}\right)} =
\im^{-1} \mathsf{s}^{-1}
\frac{\operatorname{f}(\alpha)\operatorname{S}\!
\left(e^{2\pi\mathsf{b}\alpha},e^{2\pi\mathsf{b}^{-1}\alpha}\right)}{
e^{2\pi\mathsf{s}\alpha} \operatorname{g}'
\left(e^{2\pi\mathsf{s}\alpha}\right)} \,.
$$
Combining Lemma~\ref{lem.main} with Equation~\eqref{eq.fzval} concludes
the proof of Theorem~\ref{thm.1}.
\qed

\subsection{The case of the $(-2,3,7)$ pretzel knot}
\label{sub.pretzel}

The second author computed the state-integral invariant of the $(-2,3,7)$
pretzel knot. The result of the long computation is the following 
1-dimensional state integral, which we take as input to our analysis:
\begin{equation}
\label{eq.I237}
I_{(-2,3,7)}(b)=\int_{\BR + i \epsilon} \fad{\mathsf{b}}(x)^2 
\fad{\mathsf{b}}(2x-c_{\mathsf{b}}) e^{-2 \pi \im x^2} dx
\end{equation}
The integral is absolutely convergent, and the statement and proof of 
Theorem~\ref{thm.1} applies using the following definition of the
functions $\operatorname{f}(x)$, $\operatorname{g^\pm}_k(x,q)$ and 
$\operatorname{g}(x)$:
\begin{subequations}
\begin{align}
\operatorname{f}(x-c_{\mathsf{b}}) &= 
\fad{\mathsf{b}}(x)^2 \fad{\mathsf{b}}(2x-c_{\mathsf{b}})e^{-2 \pi \im x^2} \\
\operatorname{g^\pm}_k(x,q) &=  \frac{
q^{k(k+1)} \!x^{2k} }{\left(q x;q\right)^{2}_k\!\left(q x^2;q\right)_{2k}}   \\
\operatorname{g}(x) &= \frac{x^2}{\left(1-x\right)^2\!\left(1-x^2\right)^2} \,.
\end{align}
\end{subequations}
Observe that $\operatorname{f}(z)$ is non-vanishing and absolutely 
integrable along the line $\mathcal{C}_{\lambda}$ if
\begin{equation}
-(M+N)/4<\lambda<0,
\end{equation}
and $\operatorname{f}(z+\mathsf{s}\im)\ne \operatorname{f}(z)$ if 
$\lambda$ is in general position. 

We now discuss the solutions of the gluing equations $\operatorname{g}(x)=1$ and the
matching with the set of nonabelian parabolic $\PSL(2,\BC)$ representations,
illustrating Remark~\ref{rem.parabolic}.

The equation $\operatorname{g}(x)=1$ has 6 solutions that come from two cubic equations:
\begin{equation}
\label{eq.2cubic}
\frac{z}{(1-z^2)(1-z)} = \pm 1 \,.
\end{equation}
Each triple of solutions lies in number fields $F_+$ and $F_-$ 
of discriminant $-23$ and $49$ and type $[1,1]$ and $[3,0]$ respectively.

On the other hand, there are 6 nonabelian parabolic $\PSL(2,\BC)$ 
representations of the $(-2,3,7)$ pretzel knot. These may be found using
the Ptolemy methods of~\cite{ptolemy} and their {\tt snappy} 
implementation~\cite{snappy}. An alternative method is to use 
the $A$-polynomial of the pretzel knot from~\cite{Culler:Apoly}
$$
A(m,l)=l^6 - l^5 m^8 + 2 l^5 m^9 - l^5 m^{10} - 2 l^4 m^{18} - l^4 m^{19} + 
l^2 m^{36} + 2 l^2 m^{37} + l m^{45} - 2 l m^{46} + l m^{47} - m^{55}
$$
Observe that $A(1,l)=(l-1)^3(l+1)^3$. Setting 
$(m,l)=(1+t,\pm 1+c_{\pm} t + O(t^2))$
we obtain that
$$
-6119 + 2012 c_- - 220 c_-^2 + 8 c_-^3 =0, \qquad
-6193 - 2020 c_+ - 220 c_+^2 - 8 c_+^3 =0
$$
Then, we have $F_\pm=\BQ(c_\pm)$. If $z$ is a solution to~\eqref{eq.2cubic},
let $\rho_z$ denote the corresponding nonabelian parabolic $\PSL(2,\BC)$
representation. The Rogers dilogarithm of $z$ agrees with the complex volume
of $\rho_z$, and $\operatorname{g}'(x)$ agrees with the 1-loop invariant of $\rho_z$.

Incidentally, if $z \in F_+$, a totally real field, then the corresponding
triple of elements of the Bloch group is torsion and triple of complex volumes
is given by 
$$
\left(\mathsf{e}\left(-\frac{19}{42}\right), 
\mathsf{e}\left(-\frac{13}{42}\right),
\mathsf{e}\left(\frac{11}{42}\right) \right)= 
\mathsf{e}\left(-\frac{19}{42}\right) 
\left(1, 
\mathsf{e}\left(\frac{1}{7}\right),
\mathsf{e}\left(-\frac{2}{7}\right) \right)
\,,
$$ 
where $\mathsf{e}(x)$ is given by Equation~\eqref{eq.ex}. 


\section{Proof of Theorem \ref{thm.2}}
We start by taking the logarithmic derivative of Faddeev's quantum dilogarithm
\begin{align}
\frac{\partial}{\partial x}\log \fad{\mathsf{b}}(x)
& =\int_{\mathbb{R}+\im\epsilon}
\frac{-2\im e^{-2\im xz}}{4\sinh(z\mathsf{b})\sinh(z\mathsf{b}^{-1})}
\operatorname{d}\!z \\ \notag & 
=\int_{\mathbb{R}+\im\epsilon}
\frac{ e^{-2\im xz}}{2\im \sinh(z\mathsf{b})\sinh(z\mathsf{b}^{-1})}
\operatorname{d}\!z\\ \notag &
=\int_{\mathbb{R}+\im\epsilon}
\frac{\pi\mathsf{s}e^{-2\pi\im \mathsf{s}xz}}{2\im \sinh(\pi z\mathsf{b}\mathsf{s})
\sinh(\pi z\mathsf{b}^{-1}\mathsf{s})}\operatorname{d}\!z \\ \notag &
=\int_{\mathbb{R}+\im\epsilon}
\frac{ \pi\mathsf{s}e^{-2\pi\im \mathsf{s}xz}}{2\im \sinh(\pi zM)\sinh(\pi zN)}
\operatorname{d}\!z \,.
\end{align}
After rescaling $x\mapsto \frac{x}{2\pi\mathsf{s}}$ we obtain
\begin{equation}
\label{thm1.8-1}
4\im \frac{\partial}{\partial x}
\log \fad{\mathsf{b}}\left(\frac{x}{2\pi\mathsf{s}}\right)
=\int_{\mathbb{R}+\im\epsilon}
\frac{e^{-\im xz}}{ \sinh(\pi zM)\sinh(\pi zN)}\operatorname{d}\!z\,.
\end{equation}
The integrand in \eqref{thm1.8-1}, given by the function
\begin{equation}
\operatorname{f}(z)= \frac{e^{-\im xz}}{ \sinh(\pi zM)\sinh(\pi zN)}
\end{equation}
satisfies Equation~\eqref{per} with $a=\im$ as a direct consequence of the 
equalities
\begin{equation}
\frac{\operatorname{f}(z\pm\im)}{\operatorname{f}(z)}=(-1)^{M+N} e^{\pm x} \,.
\end{equation}
Equation~\eqref{residue1} and an application of Cauchy's residue
theorem implies that
\begin{align}
\notag
\frac{2}{\pi}\left(1-(-1)^{M+N} e^{x}\right)
\frac{\partial}{\partial x}
\log \fad{\mathsf{b}}\left(\frac{x}{2\pi\mathsf{s}}\right)
&=\frac{1}{2\pi\im}\left( 
\int_{\mathbb{R}+\im\epsilon}-\int_{\mathbb{R}+\im(1+\epsilon)}\right)
\operatorname{f}(z)\operatorname{d}\!z\\ \label{eq.Phi123} &
=
S_1(z) + S_2(z) + S_3(z)\,,
\end{align}
where 
$$
S_1 = \sum_{m=1}^{M-1}\operatorname{Res}_{z=\im \frac mM} 
\operatorname{f}(z), \qquad
S_2 = \sum_{n=1}^{N-1}\operatorname{Res}_{z=\im \frac nN}
\operatorname{f}(z), \qquad
S_3 = \operatorname{Res}_{z=\im}
\operatorname{f}(z) \,.
$$
So, we have reduced the integrals to the sum of residues. Our next task is
to calculate each residue. Let us introduce $C_i$ for $i=1,2,3$ by:
\begin{align*}
C_1 &=\frac{\left(1-e^{x+\pi\im(M+N)}\right)^{\frac{M-1}{2M}}}{
D_M\left(e^{(x+\pi\im(M+N))/M}; e^{2\pi \im  N/M}\right)} \\
C_2 &=\frac{\left(1-e^{x+\pi\im(M+N)}\right)^{\frac{N-1}{2N}}}{
D_N\left(e^{(x+\pi\im(M+N))/N}; e^{2\pi \im  M/N}\right)} \\
C_3 &=\left(1-(-1)^{M+N}e^x\right)^{\frac{\im x}{2\pi\mathsf{s}^2}} 
e^{\frac{\im}{2\pi\mathsf{s}^2}\operatorname{Li}_2\left((-1)^{M+N}e^x\right)} \,.
\end{align*}

\begin{lemma}
\label{lem.SC}
For $i=1,2,3$ we have:
\begin{equation}
\label{eq.SC}
S_i = \frac{2}{\pi}\left(1-(-1)^{M+N} e^{x}\right)
\frac{\partial}{\partial x} \log C_i \,.
\end{equation}
\end{lemma}

\begin{proof}
First we compute $S_1$. Expanding in powers of $z$ around $z=0$, we have
\begin{equation}
\operatorname{f}\left(z+\im \frac mM\right) 
=\frac{(-1)^me^{mx/M}}{\pi z M\im\sin(\pi m N/M)}\left(1+\mathrm{O}(z)\right)=
\frac{(-1)^me^{mx/M}}{\pi z M\im\sin(\pi m N/M)}+\mathrm{O}(1)
\end{equation}
so that
\begin{align*}
\operatorname{Res}_{z=\im \frac mM}\operatorname{f}(z)
&=\frac{\left(-e^{x/M}\right)^m}{\pi\im M\sin(\pi m N/M)}
=\frac{2\left(-e^{x/M}\right)^m}{
\pi M\left(e^{\pi \im m N/M}-e^{-\pi \im m N/M}\right)}\\
&=\frac{-2\left(-e^{(x+\pi\im N)/M}\right)^m}{\pi M\left(1-e^{2\pi \im m N/M}
\right)}
=\frac{-2e^{m(x+\pi\im(M+N))/M}}{\pi M\left(1-e^{2\pi \im m N/M}\right)} \,.
\end{align*}
Now, by using Lemma~\ref{lemim} (see below), we calculate
\begin{multline*}
-\frac{\pi }2 \sum_{m=1}^{M-1}
\operatorname{Res}_{z=\im \frac mM}\operatorname{f}(z)=
M^{-1}\sum_{m=1}^{M-1} \frac{e^{m(x+\pi\im(M+N))/M}}{1-e^{2\pi \im m N/M}}\\
=\frac{M-1}{2M} e^{x+\pi\im(M+N)} +\left(1-e^{x+\pi\im(M+N)}\right)
\frac{\partial}{\partial x}
\log D_M\left(e^{(x+\pi\im(M+N))/M}; e^{2\pi \im  N/M}\right).
\end{multline*}
Finally observe that
$$
-\int_{-\infty}^x  \frac{e^{y+\pi\im(M+N)}}{ \left(1-e^{y+\pi\im(M+N)}\right)}
\operatorname{d}\! y
= \log \left(1-e^{x+\pi\im(M+N)}\right)
$$
This proves Equation~\eqref{eq.SC} for $i=1$. Interchanging $M$ with $N$ 
proves Equation~\eqref{eq.SC} for $i=2$.
Finally we compute $S_3$. Expanding in powers of $z$ around $z=0$, we have
\begin{align*}
(-1)^{M+N} e^{-x} \operatorname{f}(z+\im)
&=\operatorname{f}(z)
=\frac{1-\im xz +\mathrm{O}(z^2)}{
\pi z M(1+\mathrm{O}(z^2))\pi z N(1+\mathrm{O}(z^2))}\\
&=\frac{1-\im xz +\mathrm{O}(z^2)}{\pi^2 MN z^2}
=\frac{1}{\pi^2\mathsf{s}^2 z^2} 
-\frac{\im x}{\pi^2 \mathsf{s}^2 z} +\mathrm{O}(1)
\end{align*}
so that
\begin{equation*}
\operatorname{Res}_{z=\im}\operatorname{f}(z)
=\frac{(-1)^{1+M+N}\im xe^x}{\pi^2 \mathsf{s}^2}\,.
\end{equation*}
Now we calculate
\begin{align*}
2\pi \im \mathsf{s}^2\log C_3 
&=\int_{-\infty}^x\frac{(-1)^{M+N}ye^y\operatorname{d}\! y}{1-(-1)^{M+N} e^{y}}
=-\int_{-\infty}^xy\operatorname{d}\log\left(1-(-1)^{M+N}e^y\right)\\
&=-\left[y\log\left(1-(-1)^{M+N}e^y\right)\right]_{-\infty}^{x}
+\int_{-\infty}^x\log\left(1-(-1)^{M+N}e^y\right)\operatorname{d}\! y\\
&=
-x\log\left(1-(-1)^{M+N}e^x\right)+\int_{0}^{(-1)^{M+N}e^x}
\frac{\log(1-z)}{z}\operatorname{d}\! z\\
&=
-x\log\left(1-(-1)^{M+N}e^x\right)-\operatorname{Li}_2\left((-1)^{M+N}e^x\right)
\,.
\end{align*}
Equation~\eqref{eq.SC} follows for $i=3$.
\end{proof}

We now finish the proof of Theorem~\ref{thm.2}. Using
\begin{equation}
\lim_{x\to -\infty}\fad{\mathsf{b}}\!\left(x\right)=1
\end{equation}
it follows that
\begin{equation}
\label{eq.pfthm2}
\log\fad{\mathsf{b}}\!\left(\frac{x}{2\pi\mathsf{s}}\right)
=\int_{-\infty}^x \frac{\partial}{\partial y}\log\fad{\mathsf{b}}\!
\left(\frac{y}{2\pi\mathsf{s}}\right)\operatorname{d}\! y
\end{equation}
Combining the above with Equation~\eqref{eq.Phi123} and Lemma~\ref{lem.SC},
we obtain that
\begin{equation}
\fad{\mathsf{b}}\!\left(\frac{x}{2\pi\mathsf{s}}\right)=C_1C_2C_3
\end{equation}
Introduce a new variable $z$ related to $x$ by 
\begin{equation}
\frac{x}{2\pi\mathsf{s}}=\frac{z}{2\pi\mathsf{s}}-c_{\mathsf{b}} \,.
\end{equation}
In other words, we have
\begin{equation}
x=z-\pi\im(M+N) \,.
\end{equation}
Equation~\eqref{eq.pfthm2} implies that
\begin{multline}
\fad{\mathsf{b}}\!\left(\frac{z}{2\pi\mathsf{s}}-c_{\mathsf{b}}\right)
\operatorname{D}_N(e^{z/N};q_+)\operatorname{D}_M(e^{z/M};q_-)
e^{-\frac{\im}{2\pi\mathsf{s}^2}\operatorname{Li_2}(e^{z})}\\=
\left(1-e^z\right)^{\frac{M-1}{2M}+\frac{N-1}{2N}+\frac{\im (z-\pi\im(M+N))}{2\pi MN}}=
\left(1-e^z\right)^{1+\frac{\im z}{2\pi MN}}.
\end{multline}
This concludes the proof of Theorem~\ref{thm.2}.
\qed

\begin{lemma}
\label{lemim}
For any complex root of unity $q$ of order $M$, we have
\begin{equation}
\sum_{m=1}^{M-1}\frac{x^m}{1-q^m}
=\frac{M-1}{2}x^M+(1-x^M)x\frac{\partial}{\partial x}\log D_M(x; q).
\end{equation}
\end{lemma}

\begin{proof}
We calculate
\begin{align*} 
(1-x^M) x\frac{\partial}{\partial x}\log D_M(x; q)
&= (1-x^M)x\frac{\partial}{\partial x}\sum_{m=1}^{M-1}\frac mM\log(1-xq^m)\\
&= -\frac{(1-x^M)x}{M}\sum_{m=1}^{M-1}\frac{mq^m}{1-xq^m}
=-\frac{x}{M}\sum_{m=1}^{M-1}\frac{mq^m\left(1-(xq^m)^M\right)}{1-xq^m} \\
&= -\frac{x}{M}\sum_{m=1}^{M-1}mq^m\sum_{n=0}^{M-1}\left(xq^m\right)^n
=-\frac{1}{M}\sum_{n=0}^{M-1}x^{n+1} \sum_{m=1}^{M-1}mq^{m(n+1)}\\
&= -\frac{1}{M}\sum_{n=1}^{M}x^{n} \sum_{m=1}^{M-1}mq^{mn}
=-\frac{M-1}2 x^M-\frac{1}{M}\sum_{n=1}^{M-1}x^{n} \sum_{m=1}^{M-1}mq^{mn} \,.
\end{align*}
To finish the proof, we do the final calculation
\begin{equation}
\sum_{m=1}^{M-1}mq^{mn}
=\left.t\frac{\partial}{\partial t}\sum_{m=1}^{M-1}t^m\right\vert_{t=q^n}
=\left.t\frac{\partial}{\partial t}\left(\frac{1-t^M}{1-t}-1\right)
\right\vert_{t=q^n}
=\left.\frac{-Mt^M}{1-t}\right\vert_{t=q^n}=\frac{-M}{1-q^n}.
\end{equation}
\end{proof}


\appendix

\section{Some useful properties of the quantum dilogarithm}
\label{sec.QDL}

The quantum dilogarithm $\fad{\mathsf{b}}(x)$ is defined by 
\cite{Faddeev}
\begin{equation}\label{fad}
\fad{\mathsf{b}}(x)
=\frac{(e^{2 \pi \mathsf{b} (x+c_{\mathsf{b}})};q)_\infty}{
(e^{2 \pi \mathsf{b}^{-1} (x-c_{\mathsf{b}})};\tq)_\infty} \,,
\end{equation}
where
$$
q=e^{2 \pi \im \mathsf{b}^2}, \qquad 
\tq=e^{-2 \pi \im \mathsf{b}^{-2}}, \qquad
c_{\mathsf{b}}=\frac{\im}{2}(\mathsf{b}+\mathsf{b}^{-1}), \qquad 
\mathrm{Im}(\mathsf{b}^2) >0. 
$$
An integral representation is given by 
$$
\fad{\mathsf{b}}(x)=\exp\left( \int_{\mathbb{R}+\im\epsilon}
\frac{e^{-2\im xz}}{4\sinh(z\mathsf{b})\sinh(z\mathsf{b}^{-1})}
\frac{\operatorname{d}\!z}{z} \right)
$$
in the strip $|\mathrm{Im} z| < |\mathrm{Im} c_{\mathsf{b}}|$. 
Remarkably, this function admits 
an extension to all values of $\mathsf{b}$ with 
$\mathsf{b}^2\not\in\mathbb{R}_{\le 0}$. $\fad{\mathsf{b}}(x)$ is 
a meromorphic function of $x$ with
$$
\text{poles:} \,\,\, c_{\mathsf{b}} + \im \BN \mathsf{b} + \im \BN \mathsf{b}^{-1},
\qquad
\text{zeros:} \,\, -c_{\mathsf{b}} - \im \BN \mathsf{b} 
- \im\BN \mathsf{b}^{-1} \,.
$$ 
The functional equation
$$
\fad{\mathsf{b}}(x) \fad{\mathsf{b}}(-x)=e^{\pi \im x^2} \fad{\mathsf{b}}(0)^2, 
\qquad
\fad{\mathsf{b}}(0)=\left(\frac{q}{\tilde q}\right)^{\frac{1}{48}} 
=e^{\pi\im\left(\mathsf{b}^2+\mathsf{b}^{-2}\right)/24}
$$
allows us to move $\fad{\mathsf{b}}(x)$ from the denominator to the numerator
of the integrand of a state-integral.

The asymptotics of the 
quantum dilogarithm are given by~\cite[App.A]{AK}
\begin{equation}
\label{eq.as}
\fad{\mathsf{b}}(x) \sim \begin{cases}
\fad{\mathsf{b}}(0)^2e^{\pi \im x^2} & \text{when} \quad \Re(x) \gg 0 \\
1             & \text{when} \quad \Re(x) \ll 0
\end{cases}
\end{equation}

The quantum dilogarithm is a quasi-periodic function. Explicitly,
it satisfies the equations
\begin{subequations}
\begin{align}
\label{eq.bshift}
\frac{\fad{\mathsf{b}}(x+c_{\mathsf{b}}+\im \mathsf{b})}{
\fad{\mathsf{b}}(x+c_{\mathsf{b}})} 
&= \frac{1}{1-q e^{2 \pi \mathsf{b} x}} 
\\ 
\label{eq.tbshift}
\frac{\fad{\mathsf{b}}(x+c_{\mathsf{b}}+\im\mathsf{b}^{-1})}{
\fad{\mathsf{b}}(x+c_{\mathsf{b}})} &= 
\frac{1}{1-\tq^{-1} e^{2 \pi \mathsf{b}^{-1} x}} \,.
\end{align}
\end{subequations}


\bibliographystyle{hamsalpha}
\bibliography{biblio}
\end{document}